\documentclass{amsart}
\usepackage{amsfonts,amssymb,mathrsfs,amscd,textcomp,cite}
\usepackage{hyperref}
\newtheorem{theorem}{Theorem}
\theoremstyle{plain}
\newtheorem{corollary}{Corollary}

\newtheorem{example}{Example}
\newtheorem{lemma}{Lemma}
\newtheorem{proposition}{Proposition}
\newtheorem{remark}{Remark}
\numberwithin{equation}{section}
\def\CC{\mathbb C}

\begin{document}
\title[Semigroup Operator Algebras and Quantum Semigroups]{Semigroup Operator Algebras and Quantum Semigroups}
\author{Aukhadiev Marat}
\email[Aukhadiev Marat]{m.aukhadiev@gmail.com}
\author{Grigoryan Suren }
\email[Grigoryan Suren]{gsuren@inbox.ru}%
\author{Lipacheva Ekaterina}
\email[Lipacheva Ekaterina]{elipacheva@gmail.com}%

\address[Aukhadiev M.A., Grigoryan S.A. and Lipacheva E.A.]
{420066, Russia, Kazan,  \newline%
\indent Krasnoselskaya str., 51\newline%
\indent Kazan State Power Engineering University}%
\date{March, 2013}

\begin{abstract}
A detailed study of the semigroup $C^\ast$-algebra is presented. This $C^\ast$-algebra appears as a ``deformation'' of the continuous functions algebra on a compact abelian group. Considering semigroup $C^\ast$-algebras in this framework we construct a compact quantum semigroups category. Then the initial group is a compact  subgroup of the new compact quantum semigroup, the natural action of this group is described. The dual space of such $C^\ast$-algebra is endowed with Banach *-algebra structure, which contains the algebra of measures on a compact group. The dense weak Hopf *-algebra is given. It is shown that the constructed category of compact quantum semigroups can be embedded to the category of abelian semigroups.
\end{abstract}
\maketitle

\tableofcontents

\section{Introduction}
  
  There exist two different approaches to quantization -- algebraic and topological. The first one is based on deformation of universal enveloping algebras. This approach takes its roots in the early works by V.G. Drinfeld \cite{Drinfeld}.   The theory of compact quantum groups and semigroups is a part of the second approach, which started with S.L. Woronowicz's works \cite{Woronowicz2} and L.L. Vaksman and Ya.S. Soibel'man. It was shown later in \cite{RTF} that these two approaches are equivalent. Let us explain in short the process of quantization in the $C^\ast$-algebras framework. 
  
   Let $P$ be a compact semigroup, i.e. a Hausdorff compact space with continuous associative operation $(x,y)\to xy$. Denote by $C(P)$ the algebra of continuous functions on $P$. Then $C(P)$ is a commutative unital $C^\ast$-algebra, which encodes all the topological information about the space $P$. We identify $C(P\times P)$ with $C(P)\otimes C(P)$. In what follows symbol $\otimes$ denotes minimal $C^\ast$-tensor product. Define a map $\Delta\colon C(P)\to C(P)\otimes C(P)$:
  $$\Delta(f)(x,y)=f(xy).$$
   Obviously, $\Delta$ is a continuous unital *-homomorphism. The associativity of multiplication in $P$ is encoded in the following property of $\Delta$, called \emph{coassociativity}:
 $$ (\Delta\otimes\mathrm{id})\Delta=(\mathrm{id}\otimes\Delta)\Delta$$
  Thus, all the information about the compact group $P$ is encoded in a pair $(C(P),\Delta)$. 
   
   Now, let $\mathcal{A}$ be a commutative unital $C^\ast$-algebra. Then, by Gelfand-Naimark theorem, $\mathcal{A}$ is isomorphic to the algebra of all continuous functions $C(P)$ on some compact Hausdorff space $P$. If $\mathcal{A}$ is endowed with a *-homomorphism $\Delta\colon \mathcal{A}\to \mathcal{A}\otimes \mathcal{A}$, satisfying the coassociativity condition, then equation $f(xy)=\Delta(f)(x,y)$ endowes a semigroup structure on $P$.
   
    Quantization means a passage from the commutative algebra $C(P)$ to a noncommutative unital $C^\ast$-algebra $\mathcal{A}$. Due to noncommutativity this algebra does not corresppond to any semigroup. Nevertheless, algebra $\mathcal{A}$ can be regarded as an algebra of continuous functions on some virtual compact geometrical object. This object is called a \emph{quantum space}. In other words, quantum space is an object of a category, dual to the category of $C^\ast$-algebras (see \cite{Woronowicz4} for details). Usually quantum space correspoding to $C^\ast$-algebra $\mathcal{A}$ is denoted by $\mathcal{QS(A)}$ \cite{Soltan}.
    
     A unital *-homomorphism $\Delta\colon \mathcal{A}\to \mathcal{A}\otimes \mathcal{A}$ satisfying coassociativity condition is called \emph{comultiplication}. Analogously to the classical case, $\Delta$ turns quantum space $\mathcal{QS(A)}$ into a quantum semigroup. Then the algebra $\mathcal{A}$ with its comultiplication represents \emph{an algebra of functions on a quantum semigroup} \cite{Tuset}. The pair $(\mathcal{A},\Delta)$ is called \emph{a compact quantum semigroup} \cite{Vandaele3}.
      
    Compact quantum semigroup $(\mathcal{A},\Delta)$ is called \emph{commutative}, if $\sigma\Delta=\Delta$, where $\sigma\colon \mathcal{A}\otimes \mathcal{A}\to \mathcal{A}\otimes \mathcal{A}$ is a flip: $\sigma(a\otimes b)=b\otimes a$.
      
   By definition of S.L. Woronowicz \cite{Woronowicz}, the compact quantum semigroup $(\mathcal{A},\Delta)$ is a compact quantum group if each of the sets
   $$ (1\otimes \mathcal{A})\Delta(\mathcal{A}),\ (\mathcal{A}\otimes 1)\Delta(\mathcal{A})$$
    is linear dense in $\mathcal{A}\otimes \mathcal{A}$. In the classical case (when $\mathcal{A}$ is commutative) $(\mathcal{A},\Delta)$ corresponds to an ordinary group, the density conditions of the above-mentioned sets corresponds to  existense of inverse elements in $P$.
    
  At the moment the quantum groups theory is widely represented (see \cite{Woronowicz1}, \cite{Vandaele2} \cite{Vandaele3}, \cite{Podles}). Most of non-trivial examples are constructed by deformation of Hopf *-algebras generated by classical groups, and then taking a closure. Futher in each case it should be proved that the comultiplication on the Hopf algebra extends to $C^\ast$-algebra.
  
 On the other hand,  problem of constructing a compact quantum semigroup on a concrete $C^\ast$-algebra remains open. This problem is unsolved even for a well-known Toeplitz algebra. 
 
  There are not so many works devoted to the theory of quantum semigroups. The problem of giving a nontrivial structure of a compact quantum semigroups on a concrete $C^\ast$-algebra has been actively studied only the last years. For example, in paper \cite{Kawamura} the comultiplication is defined an a direct sum of Cuntz algebras.
  
  This work is devoted to the construction of compact quantum semigroups on semigroup $C^\ast$-algebras, which are generated by the ``deformed'' representations of commutative $C^\ast$-algebra of continuous functions on a compact abelian group $G$ by multiplicators on $L^2(G)$. Such algebras form a class of noncommutative $C^\ast$-algebras, close in their characteristics to $C(G)$.
   
   The paper consists of two sections. The first is devoted to properties of a $C^\ast$-algebra $C^*_{red}(S)$. Some basic facts of this section can be found in \cite{Sal}. The second section is devoted to the cunstruction and investigation of properties of compact quantum semigroup on the $C^\ast$-algebra $C^*_{red}(S)$.
   
   \section{A $C^\ast$-algebra $C^*_{red}(S)$}

 \subsection{Preliminaries}

   Let $\Gamma$ be an additive discrete torsion-free group. In what follows $S$ denotes a semigroup, which generates group $\Gamma$, i.e.  $$\Gamma=\{a-b|\ a,b\in S \}.$$
   Unless otherwise stated, we suppose that $S$ contains 0 -- a neutral element of the group $\Gamma$.  We endow $S$ with a quasi-order: $a \prec b$, if $b=a+c$, where $c\in S$. This quasi-order is called \emph{natural}. 
   
   Quasi-order is called \emph{an order}, if $a\prec b$ and $b\prec a$ imply $a=b$. Quasi-oder is an order if and only if $S$ does not contain any non-trivial group.
   
   The order on $S$ is called \emph{total}, if for any elements $a$ and $b$, $a\prec b$ or $b\prec a$. In case the natural order on $S$ is total, $\Gamma=S\bigcup (-S)$, where $-S$ is a set of oppposite elements to the elements of $S$, and $S\bigcap(-S)=\{0\}$.

   Let $Is(H)$ be a semigroup of all isometric operators on a Hilbert space $H$. 
   
   A semigroup homomorphism  $\pi\colon S\to Is(H_\pi)$ is called  \emph{an isometric representation} of semigroup $S$, where $H_\pi$ is a Hilbert space. In what follows the image if an element $a\in S$ in $Is(H_\pi)$ is denoted by $T_\pi(a)$. Let $C_\pi(S)$ be a closed $C^\ast$-subalgebra in $B(H_\pi)$, generated by operators  $T_\pi(a)$ and $T_\pi(a)^*$. 
   
   Representation $\pi\colon S\to Is(H_\pi)$ is called \emph{irreducible}, if algebra $C_\pi(S)$ is irreducible on $H_\pi$. Two representations $\pi_1\colon S\to Is(H_{\pi_1})$ and $\pi_2\colon S\to Is(H_{\pi_2})$ generate \emph{canonically isomorphic } $C^\ast$-algebras if there exists an *-isomorphism $\tau\colon C_{\pi_1}(S)\to C_{\pi_2}(S)$ such that the following diagram is commutative.
  \begin{equation}
     \begin{array}{ccccccccc}
      S& \stackrel{\pi_1}{\longrightarrow}& C_{\pi_1}(S) \\
     \downarrow\lefteqn{\pi_2}&\swarrow\lefteqn{\tau}\\
     C_{\pi_2}(S)
     \end{array} \end{equation}
      
   Two representations $\pi_1\colon S\to Is(H_{\pi_1})$ and $\pi_2\colon S\to Is(H_{\pi_2})$ are \emph{unitarily equivalent}, if there exists a unitary operator $U\colon H_{\pi_1}\to H_{\pi_2}$ such that $\pi_1=U\pi_2U^*$.
   
    A representation of semigroup $S$ on $l^2(S)$ by a shift operator $a\to T_a$ is called \emph{a regular representation}. A $C^\ast$-algebra generated by this representation is denoted $C^*_{red}(S)$.
    
   Coburn \cite{Coburn} proved that all isometric non-unitary representations of semigroup $\mathbb{Z}_+$ generate canonically isomorphic $C^\ast$-algebras. Later Douglas \cite{Douglas} and Murphy \cite{Murphy1} generalized this result for totally ordered semigroups. Then the following question seems natural: \emph{which conditions on the semigroup $S$ are equivalent to the property that any isometric representations generate canonically isomorphic $C^\ast$-algebras?} 
   
   \begin{theorem} \cite{AT}
   \label{usl} The following conditions are equivalent.
   \begin{enumerate}
   \item all isometric non-unitary representations of $S$ generate canonically isomorphic $C^\ast$-algebras,
   \item the natural order is total on $S$.
   \end{enumerate}
   \end{theorem}
     
    We call operators  of type $T_\pi(a)$ and $T_\pi(a)^*$ \emph{elementary monomials}, and finite product of elementary monomials is called \emph{a monomial}. A set of all monomials forms a semigroup, which we denote by $S_\pi^*$. A representation $\pi\colon S\to Is(H_\pi)$ is called \emph{an inverse representation} of the semigroup $S$, if $S_\pi^*$ is an inverse semigroup with respect to *-operation. Recall that a semigroup is called \emph{inverse} (or \emph{a generalized group}), if it contains with each element $x$ a unique $x^*$, called \emph{inverse} to $x$, such that  $$xx^*x=x,\ x^*xx^*=x^*.$$

   \begin{theorem}
   \label{is}  The following conditions are equivalent.
   \begin{enumerate}
   \item The natural order is total on $S$,
   \item For any non-unitary representations $\pi_1\colon S\to Is(H_{\pi_1})$ and $\pi_2\colon S\to Is(H_{\pi_2})$ there exist *-isomorphisms between semigroups $S_{\pi_1}^*$ and $S_{\pi_2}^*$.
   \end{enumerate}
   \end{theorem}
   
  In \cite{AT} it was shown that each regular representation is inverse. Therefore, \emph{there exists at least one inverse representation of $S$}. 
   
    \subsection{Algebra $C^*_{red}(S)$}\label{algebra}
   
   Recall that $$l^2(S)=\{f\colon S\to\CC| \sum\limits_{a\in S} \mid f(a)\mid^2<\infty \}$$
   is a Hilbert space with respect to scalar product
   $$(f,g)=\sum\limits_{a\in S}f(a)\overline{g(a)}.$$
   A set of functions $\{e_a \}_{a\in S}$, $e_a(b)=\delta_{a,b}$ forms an orthonormal basis in $l^2(S)$. We define a representation of $S$ by a map $a\to T_a$, where $T_ae_b=e_{a+b}$ for any $b\in S$. $C^\ast$-algebra, generated by operators $T_a$ and $T_b^*$ for all $a,b\in S$ is denoted by $C^*_{red}(S)$ and called \emph{a reduced semigroup $C^\ast$-algebra of the semigroup $S$}. We denote by $S_{red}^*$ a semigroup of monomials, generated by this representation, and by $P(S)$ an algebra of finite linear combinations of such monomials.
   
     Semigroup $S$ is a net with respect to the natural order defined in previous section: for any $a$ and $b$ in $S$ there exists $c$ such that $a\prec c$ or $b\prec c$. 
     
      \begin{lemma}
      \label{mon} Let $V$ be a monomial in $S_\pi^*$, where $\pi\colon S\to Is(H_\pi)$ is a representation. Then there exist $a,b\in S$ such that
      $$\lim\limits_{c\in S} T_\pi(c)^*V T_\pi(c)=T_\pi(a)^*T_\pi(b). $$
      
      \end{lemma}
      \begin{proof}
      Suppose $V$ is a multiple of $ \{T_\pi(a_i)^*\}_{i=1}^n$ and $\{T_\pi(b_j)\}_{j=1}^m$. Take $c=a+b$, where $a=\sum\limits_{i=1}^{n}a_i$, $b=\sum\limits_{j=1}^{m}b_j$. 
      Then relations $T_\pi(d)^*T_\pi(d)=I$ and $T_\pi(d)T_\pi(k)=T_\pi(d+k)$ imply
      $$T_\pi(c)^*VT_\pi(c)=T_\pi(a)^*T_\pi(b).$$
      Now let $c\prec d$, i.e. $d=c+k$ for some $k\in S$. Then
      $$T_\pi(d)VT_\pi(d)=T_\pi(k)^*(T_\pi(c)^*VT_\pi(c))T_\pi(k)=$$
      $$T_\pi(k)^*T_\pi(a)^*T_\pi(b)T_\pi(k)=T_\pi(a)^*T_\pi(b).$$
      And we have the required relation.
      \end{proof}
      
      Using this result we introduce the following notion. In conditions of Lemma \ref{mon}, we call $b-a\in\Gamma$ \emph{an index of monomial} $V$, denoted by $\mathrm{ind}V=b-a$.
      Note that this definition of index coincides with the same notion for a regular representation introduced in \cite{Sal}. The following statement follows from Lemma \ref{mon}.
      \begin{lemma}
      $\mathrm{ind}(V_1\cdot V_2)=\mathrm{ind} V_1+\mathrm{ind}V_2$.
      \end{lemma}

   \begin{lemma}\label{l441}
   \label{ind}Let $\mathrm{ind} V=a-b$ for a monomial $V$ in $S_{red}^*$. Then either $Ve_d=e_{d+(a-b)}$ or $Ve_d=0$.
   \end{lemma}
   \begin{proof}
   Since $T_ae_b=e_{a+b}$, and $T_a^*e_b=e_{b-a}$ if $a\prec b$ or $T_a^*e_b=0$ in other case, we have
   $$\mbox{ if } Ve_d\neq 0, \mbox{ then } T_c^*VT_c e_d=Ve_d.$$
   Passing to the limit in $S$ we obtain $T_a^*T_b e_d=V e_d$. Therefore, $Ve_d=e_{d+a-b}$.
    \end{proof}
   \begin{remark}
   One can easily verify that due to Lemma \ref{ind}, monomials with zero index are projections, with eigenvectors in a basis $e_a, a\in S$. Therefore $S_{red}^*$ is an inverse semigroup, where passing to inverse element is realized by involution.
   \end{remark}
   
   Denote by $\mathfrak{A}_c$ a closed linear space in $C_{red}^*(S)$, generated by linear combinations of monomials with index $c$. 
   \begin{lemma}
   \label{acb} \begin{enumerate}
   \item $\mathfrak{A}_c\cdot\mathfrak{A}_b\subseteq\mathfrak{A}_{c+b}$,
   \item $\mathfrak{A}_0$ is a commutative $C^\ast$-algebra
   \end{enumerate}
   \end{lemma}
   \begin{proof}
   1) Take monomials $V$ and $W$ with corresponding indices $c$ and $b$. Then $\mathrm{ind} (V\cdot W)=c+b$. Consequently, $V\cdot W \in \mathfrak{A}_{c+b}$.
   
   Statement 2) is obvious.  \end{proof}
   
   Note that $\mathfrak{A}_c=T_b^* \mathfrak{A}_0T_a $, if $c=a-b$. Indeed, for $V$ in $\mathfrak{A}_c$ we have
   $$V=T_b^*T_bVT_a^*T_a. $$
   Therefore, $$c=\mathrm{ind}V= \mathrm{ind}T_b^*+\mathrm{ind}(T_bVT_a^*)+\mathrm{ind}T_a=$$
   $$=a-b+\mathrm{ind}T_bVT_a^*.$$ Consequently, $T_bVT_a^*\in \mathfrak{A}_0$.
   
   Denote by  $P_a\colon l^2(S)\to l^2(S)$ an orthogonal projection on space $\CC e_a$, $P_af=(f,e_a)e_a$. Define \emph{a conditional expectation} $Q\colon B(l^2(S))\to B(l^2(S))$, $Q(A)=\bigoplus\limits_{a\in S} P_aAP_a$, where $A\in B(l^2(S))$.
   
   \begin{lemma}
   \label{uslov} A restriction of $Q$ to $C_{red}^*(S)$ is a conditional expectation from $C_{red}^*(S)$ to $\mathfrak{A}_0$.
    \end{lemma}
    \begin{proof}
    Since $e_a, a\in S$ are eigenvectors for any monomial with zero index, we have
    $$Q(V)=V,\mbox{ if } \mathrm{ind}V=0.$$
    Let $\mathrm{ind}V=c$. Then, by Lemma \ref{ind}  $\mathrm{ind}P_aVP_a=0$ for any $a\in S$. Therefore, $\mathrm{ind}Q(V)=0$ and finally, $Q(C_{red}^*(S))=\mathfrak{A}_0$. \end{proof}
    
    Fix an element $c$ in group $\Gamma$, and define a linear map $Q_c\colon B(l^2(S))\to B(l^2(S))$ by the following relation.
     $$ Q_c(A)=\sum\limits_{a\in S}P_{a+c}A P_a.$$
    Here $P_{a+c}=0$ if $a+c$ is not contained in $S$.
    
    \begin{lemma}
    \label{qc} Operator $Q_c$ maps the algebra $C_{red}^*(S)$ on space $\mathfrak{A}_c$.
    \end{lemma}
    \begin{proof}
   Take monomial $V$ with $\mathrm{ind}V=c$. We have $VP_a=P_{a+c}V$ for any $a\in S$. Hence, $Q_c(V)=V$. If $\mathrm{ind}V\neq c$, then $Q_c(V)=0$. Therefore, $Q_c(C_{red}^*(S))=\mathfrak{A}_c$. 
    \end{proof}
   
   For any $A$ in $C_{red}^*(S)$ we call $A_c=Q_c(A)$ a $c$\emph{-coefficient} of $A$.
   
   \begin{lemma}
   \label{uq} Following conditions are equivalent:
   \begin{enumerate}
   \item $A=0$,
   \item $Q_c(A)=0$ for all $c\in\Gamma$.
   \end{enumerate}
   \end{lemma}
   \begin{proof}
   Condition $A_c=0$ implies that $P_aAP_b=0$ for any $a,b\in S$. Equivalently, $(Ae_b,e_a)=0$ for any $a$ and $b$ in $S$. Since $\{e_a\}_{a\in S}$ forms an orthonormal basis in $l^2(S)$, we get $A=0$. The reverse follows immediately from Lemma \ref{qc}.
   \end{proof}
   The following statement is an obvious consequence. 
   
   \begin{theorem}
   \label{gr} $C^\ast$-algebra $C_{red}^*(S)$ is a graded $C^\ast$-algebra and can be formally represented in the following way
   $$C_{red}^*(S)=\bigoplus\limits_{c\in\Gamma}\mathfrak{A}_c,$$
   
   $$\mbox{ and } \mathfrak{A}_c=T_a^*\mathfrak{A}_0T_b,\mbox{ if } c=b-a,\ b,a\in S.$$
   \end{theorem}
   
   \subsection{Dynamical system, generated by the regular representation}\label{dinam}
   
   A space $C(G,C_{red}^*(S))$ of all continuous functions on a group $G$ taking values in $C_{red}^*(S)$, is a $C^\ast$-algebra with pointwise multiplication, pointwise involution and uniform norm:
   $$f^*(\alpha)=(f(\alpha))^*,\ \|f\|_\infty=\sup\limits_{\alpha\in G}\|f(\alpha)\|,\ f\in C(G,C_{red}^*(S)). $$
   
   \emph{A generalized polynomial} in $C(G,C_{red}^*(S))$ is a function of the following type
   $$P(\alpha)=\sum\limits_{i=1}^{n}\chi^{a_i}(\alpha)A_i, $$
   where $A_i\in C_{red}^*(S)$, $a_i\in\Gamma$. A set of generalized polynomials $P(G,C_{red}^*(S))$ forms a dense involutive subalgebra in the algebra $C(G,C_{red}^*(S))$. Obviously,
   $$A_i=\int\limits_{G}\chi^{-a_i}(\alpha)P(\alpha)d\mu(\alpha), $$
   where $\mu$ is a normed Haar measure of group $G$.
   
   If $\{P_n(\alpha) \}_{n=1}^\infty$ is a Cauchy sequence, which converges to $A(\alpha)$ in $C(G,C_{red}^*(S))$, then the limit
   $$\lim\limits_{n\to\infty}\int\limits_{G}P_n(\alpha)\chi^{-a}(\alpha)d\mu(\alpha)$$
   exists and equals  $a$-coefficient of $A(\alpha)$:
   $$A_a=\int\limits_{G}\chi^{-a}(\alpha)A(\alpha)d\mu(\alpha).$$
   Thus, each function $A(\alpha)$ in $C(G,C_{red}^*(S))$ is uniquely represented in a formal Fourier series
   $$A(\alpha)\simeq \sum\limits_{a\in\Gamma}\chi^{a}(\alpha)A_a $$
   with Fourier coefficients $$A_a=\int\limits_{G}\chi^{-a}(\alpha)A(\alpha)d\mu(\alpha). $$
   
   Denote by $C^*(G,S)$ a $C^\ast$-subalgebra of $C(G,C_{red}^*(S))$, generated bby elements $$\widehat{T_a}(\alpha)=\chi^a(\alpha)T_a\mbox{ and }\widehat{T_a}^*(\alpha)=\chi^{-a}(\alpha) T_a,$$
   for all $a$ in $S$.
   
   \begin{lemma}
   \label{fr} Each function $A(\alpha)\in C^*(G,S)$ is formally represented by
   $$A(\alpha)=\bigoplus\limits_{a\in\Gamma} \chi^a A_a, $$
   where $A_a$ lies in space $\mathfrak{A}_a$.  
   \end{lemma}
   \begin{proof}
   A finite product $\widehat{V}(\alpha)$ of functions $T_a(\alpha)$ and $T_b^*(\alpha)$, $a,b\in S$ can be represented in a form
   $$\widehat{V}(\alpha)=\chi^c(\alpha)V,$$
   where $V$ is a monomial, i.e. $\widehat{V}$ without $\chi$, and $c=\mathrm{ind}V$. Hence, each generalized polynomial $P(\alpha)$ in $P(G,S)=P(G,C^*(G,S))\bigcap C^*(G,S)$ is represented in a form
   $$P(\alpha)=\sum\limits_{i=1}^{n}\chi^{a_i}(\alpha)A_{a_i},$$
   where $A_{a_i}\in\mathfrak{A}_{a_i}$.  \end{proof}
   
   \begin{theorem}
   \label{din} There exists a *-isomorphism between $C^\ast$-algebras $C^*(G,S)$ and $C_{red}^*(S)$.
   \end{theorem}
   \begin{proof}
   Let $e$ be a neutral element of $G$. Define a map $\widehat{e}\colon C(G,S)\to C_{red}^*(S)$ by $\widehat{e}(A(\alpha))=A(e)$. Obviously, $\widehat{e}$ is a surjective *-homomorphism. Suppose $A(\alpha)\in Ker\widehat{e}$ and $A\neq 0$. Then 
   $$A(\alpha)=\sum\limits_{a\in\Gamma}\chi^a(\alpha)A_a,$$
   where $A_a\in\mathfrak{A}_a$. Since $A(e)=0$, by virtue of Lemma \ref{uq}, $A_a=0$ for any $a\in\Gamma$. Therefore, $A(\alpha)=0$, which shows that $\widehat{e}$ is a *-isomorphism.  \end{proof}
   
   Consider a shift $\widehat{\alpha}\colon C(G,C_{red}^*(S))\to C(G,C_{red}^*(S))$, $\widehat{\alpha}(f)=f_\alpha$, where $f_\alpha(\beta)=f(\alpha\beta)$. Obviously, it is a *-automorphism on algebra $C(G,C^*(S))$ for all $\alpha\in G$. Then one can define a representation $\tau\colon G\to Aut(C(G,C_{red}^*(S)))$ of $G$ into *-automorphism group of $C^\ast$-algebra $C(G,C_{red}^*(S))$. Due to $G$-invariance of algebra $P(G,S)$, its closure $C^*(G,S)$ is invariant as well. This implies that representation $\tau$ is also a reresentation of $G$ into group $Aut(C^*(G,S))$. Hence, a triple $(C^*(G,S),G,\tau)$ is a $C^\ast$-dynamical system. Finally, we obtain the following statement.
   
   \begin{theorem}
   \label{pred} Mapping $\tau$ is a representation of group $G$ in $Aut(C_{red}^*(S))$.
   \end{theorem}
   
   Thus, \emph{there exists a non-trivial $C^\ast$-dynamical system} $(C_{red}^*(S),G,\tau)$.
   
   \subsection{Ideals in the algebra $C_{red}^*(S)$}
   
      Consider a representation of $C(G)$ on a Hilbert space $L^2(G,d\mu)$:
      $$f\mapsto T_f, \ T_fg=f\cdot g.$$  This map realizes a *-isomorphism between $C(G)$ and a commutative subalgebra in the algebra $B(L^2(G,d\mu))$ of all bounded linear operators on $L^2(G,d\mu)$. One can easily verify that $\| T_f\|=\|f\|_\infty$.

   Denote by $H^2(S)$ subspace in $L^2(G,d\mu)$ of functions with spectrum in $S$, i.e. $f$ lies in $H^2(S)$ if and only if $f=\sum\limits_{a\in S}C_a^f\chi^a$. The Fourier transform maps $H^2(S)$ on $l^2(S)$. Let $P_S\colon L^2(G,d\mu)\to H^2(S)$ be an orthonormal projection on $H^2(S)$.
   
   \begin{lemma}
   \label{ps1} For any $f\in C(G)$ we have
   $$\|P_ST_fP_S\|=\|f\|_\infty. $$
   \end{lemma}
   \begin{proof}
    Assume that $f$ is a polynomial:
    $$f=\sum\limits_{a_i\in\Gamma,i=1}^{n}C_{a_i}^f\chi^{a_i}.$$
    Since $\|T_f\|=\|f\|_\infty$ for any $\varepsilon>0$ there exists $$g=\sum\limits_{j=1}^{m} C_{b_j}^g\chi^{b_j}$$
    in $L^2(G,d\mu)$ such that $\|g\|_2=1$ and  $$\|T_f g\|_2>\|f\|_\infty-\varepsilon. $$
    Take $a$ in $S$ such that $a+b_j$ and $a+a_i+b_j$ lie in $S$ for any $1\leq i\leq n$ and $1\leq j\leq m$. Then functions $\chi^ag$ and $\chi^afg$ lie in $H^2(S)$. Obviously, $\|\chi^ag\|_2=\|g\|_2=1$ and $$P_ST_fP_S\chi^ag=\chi^afg.$$
    Hence, $$ \|P_ST_fP_S\chi^ag\|_2=\|fg\|_2>\|f\|_\infty-\varepsilon.$$
      \end{proof}
    Note that Fourier transform maps $P_ST_{\chi^c}P_S$ to $T_c$ in $H^2(S)$, and $P_ST_{\chi^{-c}}P_S$ to $T_a^*$ in $l^2(S)$ for any $c\in\Gamma$. If $c=b-a$, where $a,b\in S$, then $$P_ST_{\chi^c}P_S=P_ST_{\chi^{-a}}P_ST_{\chi^b}P_S=T_a^*T_b,$$
   $$(P_ST_{\chi^a}P_S)^*=P_ST_{\chi^{-a}}P_S.$$
  
  In what follows for $c=b-a$ we denote by $T_c$ operator $T_a^*T_b$.
  
  Denote by $\mathbb{T}_f$ a Fourier transform of $P_ST_fP_S$, and call $f$ \emph{a symbol} of operator $\mathbb{T}_f$.
  
   \begin{corollary}\label{c1}
    Assume that $f\in C(G)$ is represented as $$f\simeq \sum\limits_{c\in\Gamma}C_c^f\chi^c.$$ Then the Fourier transform of  $P_ST_fP_S$ is represented as $$\mathbb{T}_f\simeq\sum\limits_{c\in\Gamma}C_c^fT_c\mbox{ with } \|\mathbb{T}_f\|=\|f\|\infty.$$
    \end{corollary}
   \begin{corollary}
   The $C^\ast$-subalgebra in $B(l^2(S))$, generated by operators $\mathbb{T}_f$, $f\in C(G)$ coincides with $C_{red}^*(S)$.
   \end{corollary}
     Thus, \emph{ mapping $f\to \mathbb{T}_f$ is an isometric embedding of $C(G)$ into $C_{red}^*(S)$.}
   
   Note that $\mathbb{T}_f\cdot \mathbb{T}_g=\mathbb{T}_{f\cdot g}$ if and only if the spectrum $sp(f)$ and $sp(g)$ of functions $f$ and $g$ is contained in $S$.
   
   \emph{A commutator ideal} $K$ of algebra $C_{red}^*(S)$ is an ideal, generated by commutators $T_cT_d-T_dT_c$, where $d,c\in\Gamma$. Note that $K$ is a non-unital $C^\ast$-algebra. 
   
  A restriction of representation $\tau\colon G\to Aut(C_{red}^*(S))$ onto $K$ generates representation $\tau_K\colon G\to Aut K$:
  $$\tau(\alpha)(T_cT_d-T_dT_c)=\chi^{c+d}(\alpha)(T_cT_d-T_dT_c).$$
   
   Consequently, dynamical system $(C_{red}^*(S),G,\tau)$ implies two other dynamical systems $(K,G,\tau)$ and $(C_{red}^*(S)/K,G,\tau)$, where $C_{red}^*(S)/K$ is a quotient algebra.
   
   \begin{lemma}\label{naidetsa}
    Пусть $A\in C_{red}^*(S)$. Тогда найдется такой $f\in C(G)$, что
    $$\mathbb{T}_f=\varinjlim\limits_{c\in S}T_c^*AT_c.$$
   \end{lemma}
   \begin{proof}
    Достаточно доказать это утверждение для конечной линейной комбинации мономов из $C_{red}^*(S)$. Пусть $$A=\sum\limits_{i=1}^{n}C_i^A V_i,\ C_i^A\in\CC,\ V_i\in S_{red}^*.$$
    Тогда $$\varinjlim\limits_{c\in S}T_c^*AT_c=\sum\limits_{i=1}^{n}C_iT_{d_i},$$
    где $d_i=\mathrm{ind}V_i,\ i=1,2,...n$. Отсюда согласно следствию \ref{c1},
    $$\sum\limits_{i=1}^{n}C_iT_{d_i}=\mathbb{T}_g,\ g=\sum\limits_{i=1}^{n}C_i\chi^{d_i}.$$
      \end{proof}

   \begin{lemma}
   \label{fakt} The quotient algebra $C_{red}^*(S)/K$ is isomorphic to $C(G)$.
   \end{lemma}
   \begin{proof}
   Denote by $[T_c]$ equivalence class of monomial $T_c$ by ideal $K$. Obviously,
   $$[T_c][T_d]=[T_{c+d}].$$
   Hence, for a monomial $V$ with $c=\mathrm{ind}V$, we have $[T_c]=[V]$. It implies that for $$A=\sum\limits_{i=1}^{n}C_i^AV_i, $$
   we have the following relation for equivalence classes
   $$[A]=\sum\limits_{i=1}^{n}C_i^A[T_{d_i}], $$
   where $d_i=\mathrm{ind}V_i$. Therefore, $[A]=[\mathbb{T}_f]$, $f=\sum\limits_{i=1}^{n}C_i^A\chi^{d_i}$. Finite linear combinations of monomials are dense in $C_{red}^*(S)$. Thus, for any $A\in C_{red}^*(S)$ there exists $f\in C(G)$ such that $[A]=[\mathbb{T}_f]$.
   
   Note that if $[A]=[\mathbb{T}_f]$, $[B]=[\mathbb{T}_g]$, то $[A]+[B]=[\mathbb{T}_{f+g}]$ and $[A\cdot B]=[\mathbb{T}_{f\cdot g}]$ with $\|[\mathbb{T}_f]\|\leq \|\mathbb{T}_f\|=\|f\|_\infty$. Consequently, mapping $f\to[\mathbb{T}_f]$ is a *-isomorphism between $C^\ast$-algebras $C(G)$ and $C_{red}^*(S)/K$.  \end{proof}

   \begin{lemma}
   \label{k}An element $A\in C_{red}^*(S)$ lies in $K$ if and only if
   $$\varinjlim\limits_{c\in S}T_c^*AT_c=0.$$
   \end{lemma}
   \begin{proof}
   Suppose $\varinjlim\limits_{c\in S}T_c^*AT_c=0 $ herewith $A\notin K$. Then $[A]=\mathbb{T}_f\neq 0$ for some function $f$ with the following relation $$[T_c^*AT_c]=[T_c^*T_fT_c]=[\mathbb{T}_f].$$
   The last relation follows from Lemma \ref{naidetsa}. This causes a contradiction. The reverse statement is obvious. 
   \end{proof}
   
   \begin{theorem}
   \label{ktp} A short exact sequence
   $$ 0\to K \xrightarrow{\mathrm{id}}C_{red}^*(S)\xrightarrow{\xi}C(G)\to 0,$$
   where $\mathrm{id}$ is an embedding and $\xi$ is a quotient map, splits by means of involution-preserving map $\rho\colon C(G)\to C_{red}^*(S)$ such that $\xi\circ\rho=\mathrm{id}$.
     \end{theorem}
   \begin{proof}
   Define an isometric map $\rho\colon C(G)\to C_{red}^*(S)$, which maps $f\to \mathbb{T}_f$. For any $A\in K$ we have
   $$\|\mathbb{T}_f+A\|\geq\|\mathbb{T}_f\|.$$
   Indeed, $$\|\mathbb{T}_f+A\|\geq \varinjlim\limits_{c\in S}\|T_c^*(\mathbb{T}_f+A)T_c\|=\|\mathbb{T}_f\|.$$
   Therefore, $C_{red}^*(S)=\rho(C(G))\bigoplus K$.\end{proof}

 \section{Quantization of semigroup algebras}
 
\subsection{Compact quantum semigroup}

   Let $K$ be a compact semigroup and denote by $C(K)$ an algebra of continuous functions on $K$. It was shown in introduction that this algebra is endowed with a comultiplication $\Delta$ such that $(C(K),\Delta)$ reflects the structure of a compact semigroup $K$. But not every pair $(C(K),\Delta)$ with a coassociative *-homomorphism $\Delta\colon C(K)\to C(K)\otimes C(K) $ encodes the structure of $K$. One can see this for comultiplication defined by $\Delta(f)=f\otimes 1$. This motivates us to search for a quantization of semigroup which reflects the structure and properties of initial semigroup.
   
   The main aim of this section is to put in correspondence a compact quantum semigroup for any abelian cancellative semigroup. For this purpose we shall use the theory of $C^\ast$-algebras $C^*_{red}(S)$ and results of the previous section. But in order to construct compact quantum semigroups we introduce a bit different definition.
    
 Let $\mathcal{A}$ be a $C^\ast$-алгебра, and $\mathcal{A}\odot\mathcal{A}$ algebraic tensor product. Faithful representations $\pi_1\colon\mathcal{A}\to B(H_1)$ and $\pi_2\colon \mathcal{A}\to B(H_2)$ generate representation $\pi=\pi_1\otimes\pi_2\colon \mathcal{A}\odot\mathcal{A}\to B(H_1)\otimes B(H_2)$. Closure of $\pi(\mathcal{A}\odot\mathcal{A})$ in operator norm is a tensor product $\mathcal{A}\otimes_\pi\mathcal{A}$, generated by representation $\pi$. 
 
 Consider a dual space $(\mathcal{A}\otimes_\pi\mathcal{A})^*$ to the $C^\ast$-algebra $\mathcal{A}\otimes_\pi\mathcal{A}$, i.e. a space of all linear continuous functionals on $\mathcal{A}\otimes_\pi\mathcal{A}$. For any $\xi,\eta\in (\mathcal{A})^*$ there exists a unique functional $\xi\otimes\eta$ on $\mathcal{A}\otimes_\pi\mathcal{A}$, such that
 $$(\xi\otimes\eta)(A\otimes B)=\xi(A)\eta(B),\ A,B\in\mathcal{A}.$$
 Therefore one can define $(\mathcal{A})^*\otimes_\pi(\mathcal{A})^*$ as a closure of $(\mathcal{A})^*\odot(\mathcal{A})^*$ in Banach space $(\mathcal{A}\otimes_\pi\mathcal{A})^*$. Note that $(\mathcal{A})^*\otimes_\pi(\mathcal{A})^*=(\mathcal{A}\otimes_\pi\mathcal{A})^*$ if and only if $\mathcal{A}$ is a finite-dimensional $C^\ast$-algebra.
 
 Every embedding  $\Delta\colon\mathcal{A}\to\mathcal{A}\otimes_\pi\mathcal{A}$ induces a Banach algebra structure on $\mathcal{A}^*$: \emph{ product of $\xi,\eta\in \mathcal{A}^*$ is a functional } $\xi\times\eta\in\mathcal{A}^*$ such that
 $$(\xi\times\eta)(A)=(\xi\otimes\eta)\Delta(A).$$ 
 Note that $$\|\xi\times\eta\|\leq \|\xi\|\cdot\|\eta\|.$$
 Generally, Banach algebra $\mathcal{A}^*$ need not to be associative. 
 
 If $\Delta\colon\mathcal{A}\to\mathcal{A}\otimes_\pi\mathcal{A}$ is not only an embedding but also a *-homomorphism and $\mathcal{A}^*$ is an associative Banach algebra, then we call $\Delta$ a $\pi_1\otimes\pi_2$-comultiplication, and pair $(\mathcal{A},\Delta)$ is called \emph{a compact quantum semigroup}.
 
 If any functionals commute, $\xi\times\eta=\eta\times\xi$, then quantum semigroup $(\mathcal{A},\Delta)$ is called \emph{cocommutative}. Now the coassociativity condition for $\Delta$ makes no sense: $$(I\otimes\Delta)\Delta=(\Delta\otimes I)\Delta.$$
 \begin{proposition}
 \label{p511} Let $\Delta\colon\mathcal{A}\to\mathcal{A}\otimes_\pi\mathcal{A}$ be a comultiplication. If $\pi_1',\pi_2'$ are faithful *-representations  of algebra $\mathcal{A}$, then there exists comultiplication $\Delta'\colon\mathcal{A}\to\mathcal{A}\otimes_\pi'\mathcal{A}$, where $\pi'=\pi_1'\otimes\pi_2'$.
 \end{proposition}
 \begin{proof}
 An identity map $\mathrm{id}$ on $\mathcal{A}$ induces a canonical isomorphism $I\otimes I\colon \mathcal{A}\otimes_\pi\mathcal{A}\to \mathcal{A}\otimes_\pi'\mathcal{A}$.  Define $\Delta'=(I\otimes I)\Delta$. It is sufficient to show $(\xi\otimes\eta)\Delta'=(\xi\otimes\eta)\Delta$ for any $\xi,\eta\in \mathcal{A}^*$. Indeed, we have $$(\xi\otimes\eta)\Delta'(A)=(\xi\otimes\eta)(I\otimes I)\Delta(A)=(\xi\otimes\eta)\Delta(A).$$
 \end{proof}
This proposition implies the following statement, which shows the equivalence of two definitions.
 \begin{corollary}\label{eq}
 Following conditions are equivalent:
 \begin{enumerate}
 \item There exists a (standard) comultiplication $\Delta\colon\mathcal{A}\to \mathcal{A}\otimes_{min}\mathcal{A}$
 \item For any faithful *-representations $\pi_1,\pi_2$ of the $C^\ast$-algebra $\mathcal{A}$ there exists a comultiplication $\Delta'\colon\mathcal{A}\to \mathcal{A}\otimes_\pi\mathcal{A}$, where $\pi=\pi_1\otimes\pi_2$.
 \end{enumerate}
 
 \end{corollary}

  We provide this definition with following examples.
  
   Let $C(S^1)$ denote the $C^\ast$-algebra of all continuous functions on a unit circle $S^1$, and $M(S^1)=(C(S^1))^*$ -- a space of all regular Borel measures on $S^1$.
   
   \begin{example}
   \label{prim321} Define $\Delta\colon C(S^1)\to C(S^1)\otimes C(S^1)$ by setting $\Delta(e^{in\theta})=e^{in\theta}\otimes e^{in\tau}$, $n\in\mathbb{Z}$. By virtue of Grothendieck Theorem, we have $C(S^1)\otimes C(S^1)=C(S^1\times S^1)$. Due to this equation, the map $\Delta$ has the following form:
   $$\Delta(f)(e^{i\theta},e^{i\tau})=f(e^{i\theta}\cdot e^{i\tau}).$$
   In this case a pair $(C(S^1),\Delta)$ is a compact quantum group and $M(S^1)$ is an associative Banach algebra with respect to convolution.
   \end{example}
   
   \begin{example}
   \label{prim322} Define $\Delta_2\colon C(S^1)\to C(S^1)\otimes C(S^1)$, which maps $e^{i\theta}\to e^{i2\theta}\otimes e^{i2\tau}$, i.e. $\Delta(f)(e^{i\theta,e^{i\tau}})=f(e^{i2\theta}\cdot e^{i2\tau})$. Then $(C(S^1),\Delta)$ is not a compact quantum semigroup. The space $M(S^1)$ is a non-associative Banach algebra with a zero divisor in this case.
   \end{example}
 
 Recall that a compact quantum semigroup $(\mathcal{A},\Delta)$ is called \emph{a compact quantum group}, if linear spaces $(\mathcal{A}\otimes_{min} I)\Delta(\mathcal{A})$ and $(I\otimes_{min} \mathcal{A})\Delta(\mathcal{A})$ are dense in $\mathcal{A}\otimes_{min}\mathcal{A}$.   
 
 \subsection{Quantization of the semigroup $S$}\label{copr}
 
 In this section we give a comutiplication $\Delta\colon C^*_{red}(S)\to C^*_{red}(S)\otimes C^*_{red}(S)$ and show some properties of a compact quantum semigroup $(C^*_{red}(S),\Delta)$.
 
 Cartesian product $\Gamma^2=\Gamma\times\Gamma$ of the group $\Gamma$ is endowed with a natural group structure: $(a,b)+(c,d)=(a+c,b+d)$. Obviously, semigroup $S^2=S\times S$ generates group $\Gamma$. Let $C^*_{red}(S^2)$ be a $C^\ast$-algebra on $l^2(S^2)$, generated by a regular representation of $S^2$.
 
 \begin{lemma}
 \label{s2} The following equations hold
 $$l^2(S^2)=l^2(S)\otimes l^2(S),$$
 $$C^*_{red}(S^2)=C^*_{red}(S)\otimes C^*_{red}(S).$$
 
 \end{lemma}
 \begin{proof}
 Mapping $e_{(a,b)}\to e_a\otimes e_b$ extends to a linear isometric map $U\colon l^2(S^2)\to l^2(S)\otimes l^2(S)$, taking the basis $\{e_{(a,b)} \}_{(a,b)\in S^2}$  of the space $l^2(S^2)$ to the basis $\{e_{(a,b)} \}_{(a,b)\in S^2}$ in $l^2(S)\otimes l^2(S)$.
 
 This induces an operator $T_{(a,b)}\to T_a\otimes T_b$, $(a,b)\in S^2$. Obviously,  $$T_{(a,b)}=U^* T_a\otimes T_b U.$$
 Therefore the map $T_{(a,b)}\to T_a\otimes T_b$ extends to an *-isomorphism $C^*_{red}(S^2)\to C^*_{red}(S)\otimes C^*_{red}(S)$.  \end{proof}
 
 Note that $(S_{red}^2)^*=S_{red}^*\times S_{red}^*$, where $(S_{red}^2)^*$ is an inverse semigroup generated by monomials in $C^*_{red}(S^2)$ (see section \ref{algebra}), and $P(S^2)=P(S)\odot P(S)$.
 
 We start a construction of quantum semigroup on $C^*_{red}(S)$ with its subalgebra $P(S)$. Define $\Delta\colon P(S)\to P(S)\otimes P(S)$ by relation
 $$\Delta(\sum\limits_{i=1}^{n}\lambda_i V_i)=\sum\limits_{i=1}^{n} \lambda_i V_i\otimes V_i, $$
 where $V_i$ is a monomial. 
 
 Now our aim is to show that this comultiplication extends to an embedding $\Delta\colon C^*_{red}(S)\to C^*_{red}(S)\otimes C^*_{red}(S)$. To this end, divide the basis $\{e_c\otimes e_d\}_{c,d\in S}$ of the Hilbert space $l^2(S)\otimes l^2(S)$ to equivalence classes in the following way.
 
 We say that elements $e_c\otimes e_d$ and $e_l\otimes e_k$ of the basis in $l^2(S)\otimes l^2(S)$  \emph{equivalent} ($e_c\otimes e_d \sim e_l\otimes e_k$), if $e_{c+a}\otimes e_{d+a}=e_{l+b}\otimes e_{k+b}$ for some $a,b\in S$. One can easily see that  $e_c\otimes e_d\sim e_l\otimes e_k$ and $e_l\otimes e_k\sim e_m\otimes e_n$ imply $e_c\otimes e_d\sim e_m\otimes e_n$. Note that $e_c\otimes e_d\sim e_l\otimes e_k$ if and only if $d-c=k-l$. Denote by $H_c$ a Hilbert subspace in $l^2(S)\otimes l^2(S)$, generated by family $\{e_l\otimes e_k| k-l=c,\ k,l\in S\}$. Obviously, $l^2(S)\otimes l^2(S)=\bigoplus\limits_{c\in\Gamma}H_c$. Each of the spaces $H_c$ is invariant under $C^\ast$-algebra $\Delta(P(S))$. Indeed, let $V$ be a monomial, $e_l\otimes e_k\in H_c$ and $(V\otimes V)e_l\otimes e_k\neq 0$. Due to Lemma \ref{l441}, 
 $$(V\otimes V)(e_l\otimes e_k)=e_{l+c}\otimes e_{k+c},$$
 where $c=\mathrm{ind}V$.
 \begin{lemma}\label{cred}
  Operator $\Delta\colon P(S)\to P(S)\otimes P(S)$ extends to an embedding $\Delta\colon C^*_{red}(S)\to C^*_{red}(S)\otimes C^*_{red}(S)$.
 \end{lemma}
 \begin{proof}
 
 Fixing $a\in S, Q\in P(S)$, we show the commutativity of the following diagram:

   \begin{equation}\label{qdiag}\begin{CD}
   H_a @>\rho_a>> l^2(S) \\
   @V\Delta_a(Q)=\Delta(Q)|_{H_a}VV @VVP_aQV \\
   H_a @>\rho_a>> l^2(S)
   \end{CD}
   \end{equation}
 where $\rho_a$ is a projection onto the first component:
 $$\rho_a(\sum\limits_{i=1}^{n}\lambda_i e_{a_i}\otimes e_{b_i})=\sum\limits_{i=1}^{n} \lambda_i e_{a_i}, $$
 $P_a\colon l^2(S)\to\rho_a(H_a)\subset l^2(S)$ is an orthogonal projection. Obviously, $\rho_a$ is an isometric embedding of $H_a$ into $l^2(S)$. It is sufficient to check \ref{qdiag} for basic elements $e_c\otimes e_d$ in $H_a$. Suppose $Q=\sum\limits_{i=1}^{n}\lambda_i V_i $. Then
 $$ \Delta(Q)(e_c\otimes e_d)=\sum\limits_{i=1}^{n}\lambda_i' e_{c+a_i}\otimes e_{d+a_i}, $$
 where $a_i=\mathrm{ind} V_i $ and
 $$ \lambda_i'=\left\{ \begin{array}{c}
 \lambda_i, \mbox{ if } c+a_i,d+a_i\in S\\ 0,\mbox{ if either } c+a_i\notin S,\mbox{ or } d+a_i\notin S.
 \end{array}\right. $$
 Consequently, 
 $$\rho_a(\Delta)(Q)(e_c\otimes e_d)=\sum\limits_{i=1}^{n}\lambda_i' e_{c+a_i}.  $$
 On the other hand, $$ ((P_aQ)\rho_a)(e_c\otimes e_d)=P_a(Q(e_c))=P_a(\sum\limits_{i=1}^{n}\lambda_i e_{c+a_i})=\sum\limits_{i=1}^{n}\lambda_i'e_{c+a_i}. $$
 And commutativity of diagram (\ref{qdiag}) is proved. Further, $\|\rho_a\Delta_a(Q)\| =\|P_aQ\rho_a\| $.
 Due to the fact that  $\rho_a$ is embedding and $P_a$ is projection, we have $\|\Delta_a(Q)\|\leq\|Q\|$ for all $a\in\Gamma$. 
 
 Note that the space $H_0$ is generated by $e_a\otimes e_a$ for all $a\in S$. Operator $\rho_0\colon H_0\to l^2(S)$ is a bijection and $P_0$ is an identity. Therefore, $\|\Delta_0(Q)\|=\|Q\|$ for all $Q\in P(S)$. Thus, operator $\Delta\colon P(S)\to P(S)\otimes P(S)$ extends to a continuous embedding $\Delta\colon C^*_{red}(S)\to C^*_{red}(S)\otimes C^*_{red}(S)$.
 
 \end{proof}

 \begin{corollary}\label{aea}
  Suppose $A\in C^*_{red}(S)$ and $Ae_a=\sum\limits_{i=1}^{\infty}\lambda_i e_{a_i}$. Then 
  $$ \Delta(A)(e_a\otimes e_a)=\sum\limits_{i=1}^{\infty} \lambda_i (e_{a_i}\otimes e_{a_i}). $$
 \end{corollary}
 \begin{proof}
  By virtue of definition of $\Delta$ on the subalgebra $P(S)$ of algebra $C^*_{red}(S)$, we have
  $$ \Delta(\sum\limits_{i=1}^{n}\lambda_i V_i)=\sum\limits_{i=1}^{n}\lambda_i V_i\otimes V_i, \ \sum\limits_{i=1}^{n}\lambda_i V_i\in P(S). $$
  Further, $$\sum\limits_{i=1}^{n}\lambda_i V_i e_a=\sum\limits_{i=1}^{n}\lambda_i' e_{a_i},$$
  where $\lambda_i'=\lambda_i$ in case $a+\mathrm{ind}V_i\in S$, and $\lambda_i'=0$ otherwise. This implies that $$ (\sum\limits_{i=1}^{n}\lambda_i V_i\otimes V_i)(e_a\otimes e_a)=\sum\limits_{i=1}^{n}\lambda_i' e_{a_i}\otimes e_{a_i}. $$
  Since $P(S)$ is dense in $C^*_{red}(S)$, we get the required statement. \end{proof}

 \begin{theorem}
 \label{posle} The pair $(C^*_{red}(S),\Delta)$ is a compact quantum semigroup.
 \end{theorem}
 \begin{proof}
 Let us show that multiplication induced by comultiplication $\Delta$ by Lemma \ref{cred} on a Banach algebra $(C^*_{red}(S))^*$ is accosiative. Indeed, for any monomial $V$, we have the following
 $$ (\xi\times(\eta\times\zeta))(V)=(\xi\otimes(\eta\times\zeta))\Delta(V)=$$
 $$=\xi(V)\cdot(\eta\times \zeta)(V)=\xi(V)\cdot\eta(V)\cdot \zeta(V)=((\xi\times\eta)\times\zeta)(V). $$
  \end{proof}
 
 \subsection{Weak Hopf algebra }
  In 1998 F.Li \cite{Li} defined a notion, naturally generalizing a Hopf algebra. This generalization is based on weakening an antipode condition, and called \emph{a weak Hopf algebra}. Let $A$ be an algebra with a comuliplication $\Delta$. Denote by $m$ a three-argument multiplication map in algebra $A$, $m(a\otimes b\otimes c)=abc$, the same as a product of two elements.
    
   Suppose there exists a linear map $T\colon A\to A$ such that
   \begin{equation}
   \label{wha1} m(\mathrm{id}\otimes T\otimes \mathrm{id})(\Delta\otimes\mathrm{id})\Delta=\mathrm{id}
   \end{equation}
   \begin{equation}
    \label{wha2} m(T\otimes\mathrm{id}\otimes T)(\Delta\otimes\mathrm{id})\Delta=T
    \end{equation}
    Then $A$ is called \emph{a weak Hopf algebra}, and $T$ is \emph{a weak antipode}. Note that weak Hopf algebra is not bialgebra, since a counit is not required. Obviously, $T(1)=1$. Every Hopf algebra admits a unique antipode. However, this is not true for a weak Hopf algebra. Other weak Hopf algebra properties are shown in \cite{Li1}.
    
    The analogous connection between a Hopf algebra and a group takes place here as well. Weak Hopf algebras are closely related to inverse semigroups.
    
   Consider an inverse semigroup $M$ and semigroup algebra $\CC[M]$. Algebra $\CC[M]$ admits comultiplication and weak antipode:
   $$\Delta(x)=x\otimes x,\ \Delta(\sum\limits_{i}\lambda_i x_i)= \sum\limits_{i}\lambda_i x_i\otimes x_i,$$
   $$T(x)=x^*.$$
   Clearly, equations (\ref{wha1}, \ref{wha2}) hold and  $\CC[M]$ is a weak Hopf algebra, called \emph{an inverse semigroup bialgebra}.
   
   Let $A$ be a weak Hopf algebra with a weak antipode $T$. Then a set of group-like elements in $A$ forms a regular semigroup, where inverse of any $x$ is $T(x)$. Thus, every inverse semigroup is associated  with a classical weak Hopf algebra, and every weak Hopf algebra contains an inverse semigroup. 
   
   Another example is algebra $M_n(\CC)$, generated by matrix units $E_{i,j}$. Comultiplication and weak antipode are defined on generators:
   $$\Delta(E_{i,j})=E_{i,j}\otimes E_{i,j},$$
   
   $$T(E_{i,j})=E_{j,i}.$$  
   Matrix units form an inverse semigroup. Therefore, this weak Hopf algebra turns out to be an inverse semigroup bialgebra.
   
   A.M. Vershik in \cite{Vershik1} proves result, which shows the connection between inverse semigroup bialgebras and finite-dimensional bialgebras:
   \begin{theorem}[\cite{Vershik1}]
   Semigroup bialgebra of a finite inverse semigroup $S$ with a unit, convolution multiplication and pointwise comultiplication is a semi-simple cocommutative involutive bialgebra. Conversely, every finite semi-simple cocommutative involutive bialgebra is isomorphic (as an involutive bialgebra) to a semigroup algebra of some finite inverse semigroup with unit.
   \end{theorem}
     
 Recall that algebra $C^*_{red}(S)$ can be regarded as a $C^\ast$-algebra generated by isometric operators $T_a$, $T_b^*$ for all $a,b\in S$. Since the regular representation is inverse, semigroup $\Sigma$ generated by such operators is also inverse (see \cite{AT} for details). The inverse for $T_a$ in this semigroup is $T_a^*$. Thus, we obtain the following.
  
  \begin{theorem}
  \label{hop} Compact quantum semigroup $(C^*_{red}(S),\Delta)$ contains a dense weak Hopf algebra $(\CC[\Sigma],\Delta, ^*)$.
  \end{theorem}
 This result motivates to search for a certain category of quantum semigroups which admits an analogue of the well-known fact that every compact quantum group contains a dense ${}^\ast$-Hopf algebra \cite{Woronowicz}.
 
 \subsection{Compact quantum subgroup}

 In this section we extend a notion of compact quantum subgroup from quantum groups \cite{Soltan} to quantum semigroups. We start with classical (commutative) case, the algebra of functions on a semigroup.
  
  Let $P'\subset P$ be a compact subsemigroup in a compact semigroup $P$. Then there exists a natural ${}^\ast$-epimorphism $\pi\colon C(P)\to C(P')$, which is a restriction of functions from semigroup $P$ to $P'$. Denote corresponding comultiplications on this algebras by $\Delta$ and $\Delta'$. We get the following relation:
  \begin{equation}
  \label{sub}\Delta'\pi=(\pi\otimes\pi)\Delta. 
  \end{equation}
  
  Now take compact quantum semigroups $(\mathcal{A},\Delta)$, $(\mathcal{A}',\Delta')$. Following the classical case,  we call $(\mathcal{A}',\Delta')$ \emph{a compact quantum subsemigroup} in $(\mathcal{A},\Delta)$, if there exists a ${}^\ast$-epimorphism $\pi\colon \mathcal{A}\to \mathcal{A}'$, which verifies (\ref{sub}). Actually, coassociativity of $\Delta'$ follows from   relation (\ref{sub}) and coassociativity of $\Delta$. It is sufficient to require that $\Delta'$ is a unital ${}^\ast$-homomorphism. If moreover $(\mathcal{A},\Delta)$ is a compact quantum group, then $(\mathcal{A}',\Delta')$ is a compact quantum subgroup.
  
  Closed ideal $J\subset \mathcal{A}$ is called \emph{a coideal} in a compact quantum semigroup $(\mathcal{A},\Delta)$ if 
  $$\Delta(J)\subseteq J\otimes \mathcal{A}+\mathcal{A}\otimes J.$$
  \begin{theorem}
  \label{coid} Let $J$ be a coideal in $(\mathcal{A},\Delta)$. There exists comultiplication $\Delta_J$, turning $\mathcal{A}/J$ into a compact quantum subsemigroup in $(\mathcal{A},\Delta)$ by means of quotient map $\pi\colon \mathcal{A}\to \mathcal{A}/J$. Any compact quantum subsemigroup has this form.
  \end{theorem}
  \begin{proof} One can find the proof for the quantum group case in \cite{Pinzari}. Observing that the density conditions are not necessary in semigroup case, one can repeat the same here.
  \end{proof}

  \emph{An action} of a compact quantum semigroup $(\mathcal{A},\Delta)$ on a unital $C^\ast$-algebra $B$ is a unital *-homomorphism $\delta\colon B\to B\otimes \mathcal{A}$ such that $$(\delta\otimes \mathrm{id})\delta=(\mathrm{id} \otimes\Delta)\delta. $$
  Action $\delta$ is called \emph{ergodic}, if the fixed point algebra
  $$B^\delta=\{b\in B| \delta(b)=b\otimes I\}$$
  coincides with $\mathbb{C}$.
  
  Obviously, any compact quantum subsemigroup  $(\mathcal{A}',\Delta')$ in a compact quantum semigroup $(\mathcal{A},\Delta)$ naturally acts on the algebra $\mathcal{A}$:
  $$\delta=(\mathrm{id} \otimes\pi)\Delta.$$ 
  Then the fixed point algebra $\mathcal{A}^\delta$ is called \emph{a quantum left coset space}. Space $\mathcal{A}^\delta$ with $\Delta|_\mathcal{A}^\delta\colon \mathcal{A}^\delta\to \mathcal{A}\otimes \mathcal{A}^\delta$ is \emph{a quantum left quotient space} $(\mathcal{A},\Delta)/(\mathcal{A}',\Delta')$. 
 
  Let us show that $G$ is a compact (quantum) subgroup in $(C^*_{red}(S),\Delta)$. Further $K$ denotes the commutator ideal in algebra $C^*_{red}(S)$.
 
 \begin{theorem}
 \label{subc} Ideal $K$ is a coideal in $C^*_{red}(S)$. Compact quantum semigroup $(C^*_{red}(S),\Delta)$ contains a classical compact quantum subgroup $(C(G),\Delta_G)$.
 \end{theorem}
 \begin{proof}
 By Lemma \ref{fakt}, quotient map $\pi\colon C^*_{red}(S)\to C^*_{red}(S)/K$ is a *-epimorphism on algebra $C(G)$. Here we have $\pi(T_a)=\chi^a$, and $\chi^a$ is a character of group $G$. For any $a\in S$ we have
 $$(\pi\otimes\pi)\Delta(T_a)=(\pi\otimes\pi)(T_a\otimes T_a)=\chi^a\otimes\chi^a.$$
 On the other and,  classical comultiplication on $C(G)$ acts on $\chi^a$ in the following way:
 $$\Delta_G(\chi^a)(\alpha,\beta)=\chi^a(\alpha\beta)=\chi^a(\alpha)\chi^a(\beta)=(\chi^a\otimes\chi^a)(\alpha\otimes\beta).$$
 Since chaacters generate algebra $C(G)$, we obtain
 $$(\pi\otimes\pi)\Delta=\Delta_G\pi.$$
 It is sufficient to show that $K$ is a coideal. To this end, consider monomials $V,W$. Then we have
 $$\Delta(VW-WV)=VW\otimes VW- WV\otimes WV= $$ $$=(VW -WV)\otimes VW+ WV\otimes VW-WV\otimes WV=$$ $$=(VW -WV)\otimes VW+ WV\otimes (VW-WV)\in K\otimes C^*_{red}(S)+C^*_{red}(S)\otimes K. $$
 Since $K$ is a closed ideal and $\Delta$ is a *-homomorphism, the required injection holds for all elements. \end{proof}
 
  Thus, $(C(G),\Delta_G)$ is a compact quantum subsemigroup in $(C^*_{red}(S),\Delta)$. Formally, $G\subset QS(C^*_{red}(S))$. Moreover, $QS(C^*_{red}(S))$ is endowed with an action of $G$.  
  
  \begin{proposition}
  The natural action of compact quantum subgroup $(C(G),\Delta_G)$ on the algebra $C^*_{red}(S)$ is non-ergodic and coincides with representation  $\tau\colon G\to \mathrm{Aut}C^*_{red}(S)$, defined in section \ref{dinam}. The left coset space equals $\mathfrak{A}_0$. Compact quantum quotient space $(C^*_{red}(S),\Delta)/(C(G),\Delta_G)$ is a compact (classical) quantum semigroup $(\mathfrak{A}_0,\Delta|_{\mathfrak{A}_0})$.
  \end{proposition}
  
   \begin{proof} Take $a\in S$. We have $$\tau(\alpha)(T_a)=\chi^a(\alpha)T_a.$$
   This means that $$\tau(\cdot)(T_a)\in C(G,C^*_{red}(S))=C^*_{red}(S)\otimes C(G).$$
   Therefore, $$\tau(\cdot)(\cdot)\colon C^*_{red}(S)\to C^*_{red}(S)\otimes C(G), $$
  herewith $\tau(\cdot)(T_a)=T_a\otimes \chi^a$.  On the other hand,
  $$\delta(T_a)=(id\otimes\pi)(T_a\otimes T_a)=T_a\otimes\chi^a.$$
  As noted above, $\tau(\alpha)(V)=V$ for $V\in\mathfrak{A}_0$. Therefore the space $\mathfrak{A}_0$ is a left coset space with respect to the action $\delta$.\end{proof} 
 
 \subsection{Dual algebra and Haar functional}
 
 In the previous section we endowed the dual space $(C^*_{red}(S))^*$ with an associative Banach algebra structure. By cocommutativity of $\Delta$, this algebra is commutative. Futher recall that Haar functional is a state $h\in \mathcal{A}^*$, subject to relation $$h\times\phi=\phi\times h=\lambda_\phi\cdot h$$ for all $\phi\in\mathcal{A}^*$ herewith $\lambda_\phi\in\CC$.
 \begin{theorem}
 \label{shaar} Haar functional in $(C^*_{red}(S))^*$ is defined by relation $$h(A)=(Ae_0,e_0)$$ for all $A\in C^*_{red}(S)$.
 \end{theorem}
 \begin{proof}
   One can easily verify that for any monomial $V$ in $C^*_{red}(S)$, the scalar $(Ve_0,e_0)$ is not equal to zero only in case $V=I$. Therefore, for any $\phi\in\mathcal{A}^*$ we have
   $$h\times\phi(V)=h(V)\cdot \phi(V)=\left\{ \begin{array}{cc}
   0,&V\neq I,\\ \phi(I),&V=I
   \end{array}\right. $$
   It follows that $h\times\phi=\phi(I)\cdot h$. \end{proof}

 Equation $C^*_{red}(S)=K\bigoplus \varrho(C(G))$ (see Theorem \ref{ktp}) implies
 $$(C^*_{red}(S))^*=K^\perp\bigoplus\varrho(C(G))^\perp.$$

Define mappings $$\Phi_c\colon \mathfrak{A}\to \mathfrak{A},\ A\mapsto T_c^*AT_c,\ c\in S$$
 $$\Phi_c^*\colon \mathfrak{A}^*\to\mathfrak{A}^*,\ (\Phi_c^*\xi)(A)=\xi(\Phi_c(A)).$$
 Obviously, these maps are continuous. 
 
 \begin{lemma}
For any $c\in S$ operator $\Phi_c^*$ is identical on $K^\perp$, and $\Phi_c^*\colon (\varrho(C(G)))^\perp\to (\varrho(C(G)))^\perp$.
 \end{lemma}
 \begin{proof}
  Since $K$ is ideal, $\Phi_c\colon K\to K$. Any monomial $V$ in $\varrho(C(G))$ has a form $T_a^*T_b$ for some $a,b\in S$.  Therefore,
  $$\Phi_c(V)=T_c^*T_a^*T_bT_c=T_a^*T_b=V.$$ 
  Consequently, $\xi\in K^\perp$ if and only if $\Phi_c^*(\xi)=\xi$.
 \end{proof}

 Obviously, $(\varrho(C(G)))^\perp$ is an ideal in $\mathfrak{A}^*$ with respect to mulitplication $\times$. Denote by $M(G)$ the space of regular Borel measures on group $G$, dual of $\Gamma$.
 
 \begin{theorem}
 \label{kort} $K^\perp$ is a Banach algebra isomorphic to $M(G)$. The following exact sequence splits.
   $$ 0\rightarrow (\varrho(C(G)))^\perp \rightarrow \mathfrak{A}^*\rightarrow K^\perp=M(G)\rightarrow 0$$
  
 \end{theorem}

 \subsection{Category of quantum semigroups}
 
 Consider a category $\mathcal{S}_{ab}$ of discrete abelian cancellative semigroups, with arrows (morphisms) being semigroup morphisms. Denote by $\mathcal{QS}_{red}$ a category of compact quantum semigroups $(C^*_{red}(S),\Delta)$ constructed above, for all semigroups $S\in Obj(\mathcal{S}_{ab})$, herewith unital compact quantum semigroup morphisms. Recall that a *-homomorphism $\pi\colon C^*_{red}(S_1)\to C^*_{red}(S_2)$ is called \emph{ a unital morphism of the corresponding compact quantum semigroups}, if it is unital and satisfies following diagram.
 
   \begin{equation}\label{mor}\begin{CD}
    C^*_{red}(S_1) @>\pi>> C^*_{red}(S_2) \\
    @V\Delta_1 VV @VV\Delta_2 V \\
    C^*_{red}(S_1)\otimes C^*_{red}(S_1) @>\pi\otimes\pi >> C^*_{red}(S_2)\otimes C^*_{red}(S_2)
    \end{CD}
    \end{equation}
 Thus, morphisms in category $\mathcal{QS}_{red}$ are compact quantum semigroup morphisms. 

 \begin{lemma}
  \label{isom} Suppose that $A$ is an isometry in $C^*_{red}(S)$, such that $\Delta(A)=A\otimes A$. Then $A=T_c$, where $c\in S$. 
  \end{lemma}

 \begin{proof}
 Let us show that  the set $\{e_a\}_{a\in S}$ is $A$-invariant. Допустим, 
 $$Ae_a=\sum\limits_{i=1}^{\infty}\lambda_i e_{a_i}.$$ 
 Then we have
 $$ (A\otimes A)(e_a\otimes e_a)=Ae_a\otimes Ae_a=\sum\limits_{i,j=1}^{\infty}\lambda_i\lambda_j e_{a_i}\otimes e_{a_j}.$$
 On the other hand, by Corollary \ref{aea}, 
 $$ \Delta(A)(e_a\otimes e_a)=\sum\limits_{i=1}^{\infty}\lambda_i(e_{a_i}\otimes e_{a_i}). $$
 The family $\{e_a\otimes e_b \}_{a,b\in S}$ forms an orthonormal basis in $l^2(S)\otimes l^2(S)$. Therefore, $\lambda_i\lambda_j=0$ for $i\neq j$. Consequently, all $\lambda_j=0$ except for one. further, $Ae_a=\lambda e_b$. By the same reasoning, we obtain $\lambda=\lambda^2$. It follows that $\lambda=1$.
 
 Assume that $Ae_a=e_b$ with $b=c-a$. Let us verify $Ae_d=e_{d+c}$ for any $d\in S$. Suppose that $Ae_d=e_{d+k}$, $k\neq c$. Element  $e_a\otimes e_d$ lies in Hilbert space $H_l$, where $l=d-a$, i.e. the space generated by family $\{ e_{a_i}\otimes e_{b_i}  \}_{i=1}^\infty$ such that $b_i-a_i=l$. Since $\Delta(A)(e_a\otimes e_d)$ lies in $H_l$, we obtain
 $$\Delta(A)(e_a\otimes e_d)=e_{a+c}\otimes e_{d+k},$$
 with $(d+k)-(a+c)=l$. Consequently, $k=c$. Since $A$ is an isometry, $c\in S$ and $A=T_c$. \end{proof}

 \begin{theorem}\label{cat}
 There exists an injective functor which maps $\mathcal{QS}_{red}$ to the category $\mathcal{S}_{ab}$.
 \end{theorem}
 \begin{proof}
 By commutativity of (\ref{mor}), we have $\Delta_2(\pi(T_a))=\pi(T_a)\otimes\pi(T_a)$. Operators $T_a$ and $\pi(T_a)$ are isometric. By Lemma \ref{isom}, $\pi(T_a)=T_{a'}$ for some $a'\in S$. Thus, $\pi$ generates a homomorphism $\phi\colon S_1\to S_2$, $a\mapsto a'$.  \end{proof}
 
 An automorphism  $\sigma$ of  algebra $C^*_{red}(S)$ is called \emph{a compact quantum semigroup automorphism} of $(C^*_{red}(S),\Delta)$, if $(\sigma\otimes\sigma)\Delta=\Delta\sigma$.
 
 Denote by $Aut((C^*_{red}(S),\Delta))$ the automorphism group of the compact quantum semigroup $(C^*_{red}(S),\Delta)$.
 
 \begin{corollary}
 $Aut((C^*_{red}(S),\Delta))$ is isomorphic to automorphism group $Aut(S)$. 
 \end{corollary}
 
 Note that for non-negative real numbers $\mathbb{R}_+$ and non-negative rational numbers $\mathbb{Q}_+$ semigroups every morphism is an automorphism, and $Aut(\mathbb{R}_+)=\mathbb{R}$, $Aut(\mathbb{Q}_+)=\mathbb{Q}$. For non-negative integers $\mathbb{Z}_+$ we have that group $Aut(\mathbb{Z}_+)$ is trivial, and the morphism semigroup matches $\mathbb{Z}_+$.
 
 \subsection{Different comultiplications on the Toeplitz algebra}
 
 The structure of a compact quantum semigroup on the Toeplitz algebra $\mathcal{T}$ is given in \cite{AGL2}, denoted by $(\mathcal{T},\Delta)$. One can ask if there exists another comultiplication on $\mathcal{T}$. 
 
 The Toeplitz algebra is generated by the regular isometric representation of semigroup $\mathbb{Z}_+$, so $\mathcal{T}=C^*_{red}(\mathbb{Z}_+)$. Obviously, the comultiplication on $C^*_{red}(\mathbb{Z}_+)$ constructed in section \ref{copr} matches the comultiplication on the Toeplitz algebra. However, $\mathbb{Z}_+$ is not unique as an abelian semigroup which regular representation generates algebra $\mathcal{T}$.
 
 Consider semigroup $S=\{0,2,3,4,...\}$. Obviously, $S$ generates group $\mathbb{Z}$: $$\mathbb{Z}=S-S.$$ Then $C^*_{red}(S)$ is the Toeplitz agebra. Indeed, the following combination of monomials represents a right-shift operator on $l^2(S)$:
 $$T=(I-T_3^*T_2T_2^*T_3)T_2+T_2^*T_3.$$
 Futher, $(C^*_{red}(S),\Delta_S)$ is a compact quantum semigroup with $\Delta_S$ defined in Section \ref{copr}. But this comultiplication $\Delta_S$ differs from comultiplication $\Delta$:
 $$\Delta(T)=T\otimes T,$$ 
 $$\Delta_S(T)=T_2\otimes T_2 - T_3^*T_2T_2^*T_3T_2\otimes T_3^*T_2T_2^*T_3T_2+T_2^*T_3\otimes T_2^*T_3.$$
 
 Thus, the compact quantum semigroups $(C^*_{red}(S),\Delta_S)$ and $(\mathcal{T},\Delta)$ do not match. And the results of this paper show, that generally, the Toeplitz algebra can be given at least as many structures of compact quantum semigroups, as many semigroups generating $\mathbb{Z}$ exist. 
 
 \begin{lemma}
 A morphism taking $(C^*_{red}(S),\Delta_S)$ to $(\mathcal{T},\Delta)$ does not exist.
 \end{lemma}
 \begin{proof}
  Suppose, $\pi$ is a morphism of the corresponding compact quantum semigroups. Then, by theorem \ref{cat}, there exists  a morphism $\phi\colon S\to\mathbb{Z}_+$. Take $m=\phi(2)$, $k=\phi(3)$. We have
  $$m+m+m=\phi(2+2+2)=\phi(3+3)=k+k.$$
  Therefore, $m$ is even and $k>m$. Consequently, $\pi(T_2^*T_3)$ is an isometric operator, which is not true for $T_2^*T_3$. This means that $\pi(I-T_3^*T_2T_2^*T_3)=0$, and the commutator ideal (the idea of compact operators) is a kernel of homomorphism $\pi$.
 \end{proof}
 
\vskip 0.5cm This research was partially supported by RFBR grant no 12-01-97016.
 
\bibliographystyle{amsalpha}\bibliography{marat}

\end{document}